\documentclass[10pt]{article}

\usepackage{geometry}
\usepackage{color}
\usepackage{float}
\usepackage{textcomp}
\usepackage{amsthm}
\usepackage{amsmath}
\usepackage{amssymb}
\usepackage{changes}

\usepackage{mathtools}

\newtheorem{theorem}{Theorem}[section]
\newtheorem{lemma}[theorem]{Lemma}

\DeclareMathOperator{\argmin}{argmin}

\newcommand{\R}{\mathbb{R}}

\newcommand{\Dom}{\mathrm{Dom}}

\newcommand{\inner}[2]{\langle{#1},{#2}\rangle}
\newcommand{\Inner}[2]{\left\langle{#1},{#2}\right\rangle}
\newcommand{\norm}[1]{\|#1\|}

\newcommand{\tos}{\rightrightarrows}

\newcommand{\Z}{\mathcal{Z}}

\newcommand{\V}{\mathcal{V}}

\newcommand{\vgap}{\vspace{.1in}}

\newcommand{\tz}{\tilde z}

\newcommand{\bi}{\begin{itemize}}
\newcommand{\ei}{\end{itemize}}
\newcommand{\ba}{\begin{array}}
\newcommand{\ea}{\end{array}}

\begin{document}

\title{Iteration-complexity analysis  of a generalized  alternating  direction method of multipliers}

\author{V.A. Adona \thanks{Institute of Mathematics and Statistics, Federal University of Goias, Campus II- Caixa
    Postal 131, CEP 74001-970, Goi\^ania-GO, Brazil. 
    (E-mails: {\tt vandomat32@hotmail.com}, {\tt maxlng@ufg.br} and {\tt jefferson@ufg.br}).  The work of these authors was
    supported in part by  CNPq Grant  444134/2014-0, 309370/2014-0  and FAPEG/GO.}
    \and
    M.L.N. Gon\c calves \footnotemark[1]
    \and
      J.G. Melo \footnotemark[1]
  }


\maketitle

\begin{abstract}

This paper analyzes  the iteration-complexity  of a generalized  alternating  direction method of multipliers (G-ADMM)  for solving   linearly constrained convex problems. This ADMM variant, which was first proposed by {Bertsekas and Eckstein},
introduces a relaxation parameter $\alpha \in (0,2)$ into the second ADMM subproblem.
Our approach is to show that the G-ADMM is an instance of a hybrid proximal extragradient framework with some special properties, and, as a by product, we obtain ergodic  iteration-complexity for the G-ADMM with   $\alpha\in (0,2]$, improving and  complementing related results in the literature.
Additionally, we also present pointwise iteration-complexity  for the G-ADMM.
\\[2mm]
 
 \end{abstract}

%
%
%
%
%
%
%
%
%
%
%
%
%
%
%
%
%
%
%
%
%
%
%
%
%

\pagestyle{plain}

\section{Introduction} \label{sec:int}
This paper considers  the following  linearly constrained convex  optimization problem 
\begin{equation} \label{optl}
\min \{ f(x) + g(y) : A x + B y =b, \; x\in \R^n, y \in \R^p \}
\end{equation}
where $f: \R^{n} \to \R$ and $g:\R^{p} \to \R$ are convex functions,  
$A\in \R^{m \times n}$,  $B \in \R^{m\times p}$ and  $b \in \R^m$.
Problems with separable structure such as \eqref{optl} arises in many applications 
areas, for instance, machine learning, compressive sensing and image processing.
One popular method for solving  \eqref{optl}, taking advantages of its special structure,  is the alternating direction method of multipliers (ADMM)  \cite{0352.65034,0368.65053}; for detailed reviews, see  \cite{Boyd:2011,glowinski1984}.
Many variants of it have been considered  in the literature; see, for example, \cite{ChPo_ref2,Deng1,PADMM_Eckstein,GADMM2015,MJR,Gu2015,Hager,HeLinear,Lin,LanADMM}. 
The ADMM variant  studied here is  the generalized ADMM \cite{MR1168183} (G-ADMM) with proximal terms, described as follows:
given $(x_{k-1},y_{k-1},\gamma_{k-1})$  compute $(x_{k},y_{k},\gamma_{k})$~as
\begin{align} \nonumber
x_k &\in \argmin_{x} \left \{ f(x) - \inner{ \gamma_{k-1}}{Ax} +
\frac{\beta}{2} \| Ax+B y_{k-1} -b \|^2 +
\frac{1}{2} \| x-x_{k-1} \|^2_{H_1} \right\}, \\ \nonumber 
y_k &\in  \argmin_y \left \{ g(y) - \inner{\gamma_{k-1}}{By} +
\frac{\beta}{2} \| \alpha(Ax_k + B y_{k-1} - b) +B(y-y_{k-1})\|^2+\frac{1}{2} \| y-y_{k-1} \|^2_{H_2} \right\},\\ \label{ADMMclass}
\gamma_k &= \gamma_{k-1}-\beta\left[\alpha(Ax_k+By_{k-1}-b)+B(y_{k}-y_{k-1})\right] 
\end{align}
where $\beta > 0$ is a fixed penalty parameter, $(H_1,H_2)\in \R^{n \times n}\times\R^{p \times p}$
are symmetric and positive semi-definite matrices, $\alpha\in(0,2]$ is a relaxation factor and  $\|\cdot\|_{H_i}^{2}:=\inner{H_i(\cdot)}{\cdot}$, $i=1,2$.
Different ADMM  variants studied in the literature can be seen as particular instances of the G-ADMM by appropriately choosing the  matrices $H_{i}$ ($i=1,2$) and  the relaxation parameter $\alpha$.
By setting $(H_1,H_2)=(0,0)$ and $\alpha=1$, the G-ADMM  reduces to the standard ADMM. The use of over-relaxation parameter ($\alpha>1$) in some applications can accelerate  the standard ADMM; see, for instance, \cite{Ber1,MR1256136}. 
By choosing $(H_1, H_2)=(\tau_1 I_{n}-\beta A^{*}A$, $\tau_2 I_{p}-\beta B^{*}B$) for some $\tau_1 \ge \beta \|A\|^2$ , $\tau_2 \ge \beta \|B\|^2$ ($^*$ stands for the adjoint operator), the G-ADMM subproblems may become much easier to solve, since the quadratic terms involving $A^{*}A$ and $B^{*}B$ vanish;
see, for example, \cite{Deng1,Wang2012,Yang_linearizedaugmented} for discussion.
It is well-known  that an optimal solution $(x^*,y^*)$  for problem \eqref{optl} can be obtained by finding a solution  $(x^*,y^*,\gamma^*)$  of the following Lagrangian system
\begin{equation} \label{LagInclusion_Intro}
0\in  \partial f(x)- A^{*}\gamma,\quad  0\in  \partial g(y)- B^{*}\gamma,\quad Ax+By-b=0,
\end{equation}
where $\gamma^*$ is an associated Lagrange multiplier.

In this paper, we are interested in analyzing iteration-complexity of the G-ADMM to obtain an ``approximate solution" of the Lagrangian system~\eqref{LagInclusion_Intro}.
Specifically, for a given tolerance $\varepsilon>0$, we show that in at most $\mathcal{O}(1/\varepsilon)$ iterations of  the G-ADMM, we obtain, in the ergodic sense,  an ``$\varepsilon$-approximate" solution $(\hat x, \hat y, \hat\gamma)$  and a residual $\hat v=(\hat v_1,\hat v_2,\hat v_3)$ of \eqref{LagInclusion_Intro} satisfying
\[
\hat v_1 \in  \partial_{\varepsilon_1} f(\hat x)- A^{*}\hat \gamma,\quad  \hat v_2 \in \partial_{\varepsilon_2} g(\hat y)- B^{*}\hat \gamma, \quad  \hat v_3= A\hat x+B\hat y-b,\quad \|\hat v\|_{(H_1,H_2)}\leq\varepsilon,
\quad \varepsilon_1+\varepsilon_2\leq\varepsilon,
\]
where the symbol  $\partial_{\varepsilon}$ stands for  $\varepsilon-$subdiferential, and  $\| \cdot \|_{(H_1,H_2)}$ is a norm (seminorm) depending on the matrices $H_1$ and $H_2$.
Our approach is to show that the G-ADMM is an instance of a hybrid proximal extragradient (HPE) framework (see \cite{monteiro2010complexity,Sol-Sv:hy.ext}) with a very special property, namely,  a key parameter sequence $\{\rho_k\}$ associated to the sequence generated by the method is upper bounded by a multiple of $d_0$ (a parameter measuring, in some sense, the distance of the initial point to the solution set). This result is essential to obtain the ergodic iteration-complexity of the G-ADMM with relaxation parameter $\alpha \in (0, 2]$. Additionally,  we also present pointwise iteration-complexity for the G-ADMM with $\alpha \in (0, 2)$.

Convergence rates of the G-ADMM and related variants have been studied by many authors in different contexts. In \cite{GADMM2015}, the authors obtain  pointwise and ergodic convergence 
rate bounds for the G-ADMM  with $\alpha \in (0,2)$.  Paper \cite{nishihara2015general}  studies linear convergence of the G-ADMM under additional  assumptions. Some strategies are also proposed in order to choose the relaxation and penalty parameters. Linear convergence of the G-ADMM is also studied in \cite{tao2016optimal} on a general setting. Paper \cite{sun2016analysis} studies the  G-ADMM as a particular case of a general scheme in a Hilbert space and measures, in an ergodic sense, a ``partial" primal-dual gap associated to the augmented Lagrangian of problem \eqref{optl}.  Paper \cite{Et.Cor;X.Yuan2014} studies convergence rates of a generalized proximal point algorithm and obtains, as a by product, convergence rates of the particular instance of the G-ADMM in which  $(H_1,H_2)=(0,0)$. It is worth mentioning that the previous  ergodic convergence results for the G-ADMM are not focused in solving \eqref{LagInclusion_Intro} approximately in the sense of our paper.  
Iteration-complexity study of the standard ADMM and some  variants in the setting of the HPE framework have been considered in \cite{goncalves2017pointwise,MJR,monteiro2010iteration}.  Finally, convergence rates of   ADMM variants using  a different approach have been studied in  \cite{Cui,Deng1,MJR,Hager,HeLinear,He2015,Lin,nishihara2015general,LanADMM}, to name just a few.
\\

\noindent{\bf Organization of the paper.}  Section~\ref{sec:basandHPE} is divided into two subsections, Subsection~\ref{sec:bas}  presents our notation and basic results. Subsection~\ref{sec:smhpe} is devoted to the study of a modified  HPE framework and present its main iteration-complexity results  whose proofs are given in Section~\ref{sec:hpe_Analysis}. Section~\ref{sec:proximal ADMM_proof} is divided into three subsections. Subsection~\ref{sub:genr:ADMM} formally describes the generalized ADMM and Subsection~\ref{subsec:ADMM&HPE} contains some auxiliary results. 
The pointwise and ergodic  iteration-complexity results for the G-ADMM are given in Subsection~\ref{PointErg}. 
\section{Preliminary results} \label{sec:basandHPE}
This section is divided into two subsections: The first one  presents our notation and basic results, and  the second one   describes a modified  HPE framework and present its iteration-complexity bounds.

\subsection{Notation and basic definitions} \label{sec:bas}

This subsection presents some definitions, notation and basic results used in this paper.

Let  $\V$ be a finite-dimensional  
real vector space with inner product and associated norm denoted by $\inner{\cdot}{\cdot}$ and $\|\cdot\|$, respectively.
For a given  self-adjoint positive semidefinite linear operator $Q:\V\to \V$, 
the seminorm induced by $Q$ on $\V$ is defined by $\|\cdot\|_{Q}= \langle Q (\cdot), \cdot\rangle ^{1/2}$.
Since $\langle Q (\cdot), \cdot\rangle$ is symmetric and bilinear, for all $v,\tilde{v}\in\V$, we have 
\begin{equation}\label{fact}
2\left\langle Qv,\tilde{v}\right\rangle\leq \norm{v}_{Q}^{2}+\norm{\tilde{v}}_{Q}^{2}.
\end{equation}

Given a set-valued operator $T:\V\tos \V$, its domain and graph  are defined, respectively, as
\[  
\mbox{Dom}\,T:=\{v\in \V\,:\, T(v)\neq \emptyset\} \qquad \mbox{and}\qquad  
Gr(T)=\{ (v,\tilde{v})\in \V\times \V\;|\; \tilde{v} \in T(v)\}.
\]
The operator $T$ is said to be   monotone if 
\[
\inner{u-v}{\tilde{u}-\tilde{v}}\geq 0\qquad \forall \;  (u,\tilde{u}),\, (v,\tilde{v}) \,\in \, Gr(T).
\]
Moreover, $T$ is maximal monotone if it is monotone and there is no other monotone operator $S$ such that $Gr(T)\subset Gr(S)$.
Given a scalar $\varepsilon\geq0$, the 
 {$\varepsilon$-enlargement} $T^{[\varepsilon]}:\V\tos \V$
 of a monotone operator $T:\V\tos \V$ is defined as
\begin{align}
\label{eq:def.eps}
 T^{[\varepsilon]}(v)
 :=\{\tilde{v}\in \V\,:\,\inner{\tilde{v}-\tilde{u}}{v-u}\geq -\varepsilon,\;\;\forall (u,\tilde{u})\in Gr(T)\} \quad \forall\, v \in \V.
\end{align}
The 
{$\varepsilon$-subdifferential} of a 
 proper closed convex function $f:\V\to  [-\infty,\infty]$
is defined by
\[
\partial_{\varepsilon}f(v):=\{u\in \V\,:\,f(\tilde{v})\geq f(v)+\inner{u}{\tilde{v}-v}-\varepsilon,\;\;\forall \,\tilde{v}\in \V\} \qquad\forall\, v\in \V.
\]
When $\varepsilon=0$, then $\partial_0 f(v)$ 
is denoted by $\partial f(v)$
and is called the {subdifferential} of $f$ at $v$.
It is well known that the subdifferential  operator  of a proper closed convex function is maximal monotone~\cite{Rockafellar}.

The next theorem is a consequence of the transportation formula in \cite[Theorem 2.3]{Bu-Sag-Sv:teps1} combined with 
\cite[Proposition 2(i)]{Bu-Iu-Sv:teps}.
\begin{theorem}\label{for:trans}
Suppose $T:\V\tos \V$ is maximal monotone and let $\tilde{v}_i, v_i \in\V$, for $i = 1,\cdots,k$, be
such that $v_i \in T(\tilde{v}_i)$ and define
\[
\tilde{v}_{k}^{a}=\frac{1}{k}\sum_{i=1}^{k}\tilde{v}_i, \qquad v_{k}^{a}=\frac{1}{k}\sum_{i=1}^{k}v_i,\qquad 
\varepsilon_{k}^{a}=\frac{1}{k}\sum_{i=1}^{k}\inner{v_i}{\tilde{v}_i-\tilde{v}_{k}^{a}}.
\]
Then, the following hold:
\begin{itemize}
\item[(a)] $\varepsilon_{k}^{a}\geq 0$ and $v_{k}^{a}\in T^{[\varepsilon_{k}^{a}]}(\tilde{v}_{k}^{a})$;
\item[(b)] if, in addition, $T = \partial f$ for some proper closed and convex function $f$, then 
$v_{k}^{a}\in \partial_{\varepsilon_{k}^{a}}f(\tilde{v}_{k}^{a})$.
\end{itemize}
\end{theorem}
\subsection
{A HPE-type framework} \label{sec:smhpe}

This subsection describes the modified HPE framework and its corresponding pointwise and ergodic  iteration-complexity bounds.  

Let $\Z$ be a finite-dimensional real vector space 
with inner product and induced norm denoted by $\inner{\cdot}{\cdot}$ and $\|\cdot\|=\sqrt{\inner{\cdot}{\cdot}}$, respectively.
Our problem of interest in this section is the monotone inclusion problem (MIP)
\begin{align}\label{eq:inc.p}
 0\in T(z)
\end{align}
where 
$T:\Z\tos \Z$ is a maximal monotone operator. We  assume that  the solution set  of~\eqref{eq:inc.p}, denoted by $T^{-1}(0)$, is nonempty.

We now state the modified HPE framework  for solving \eqref{eq:inc.p}.

\vgap
\vgap

\noindent
\fbox{
\begin{minipage}[h]{6.4 in}
{\bf A modified HPE framework for solving \eqref{eq:inc.p}}.
\\[2mm]
(0) Let $z_0 \in \Z$, $\eta_0 \in \R_{+}$, $\sigma \in [0, 1]$ and a self-adjoint 
 positive semidefinite linear operator   $M: \Z \to \Z $ be given, and set $k=1$;
 \\[2mm]
(1) obtain $(z_k,\tilde{z}_k,\eta_k) \in \Z \times \Z \times \mathbb{R}_{+}$   such that 
\begin{align}
 M (z_{k-1}-z_{k})  &\in T(\tz_k), \label{breg-subpro} \\[2mm]      
 \|{\tz}_k-z_{k}\|^2_{M} +\eta_k &\leq \sigma \|{\tz}_k-z_{k-1}\|_{M}^{2}+\eta_{k-1}; \label{breg-cond1}
\end{align}
(2) set $k\leftarrow k+1$ and go to step 1.
\noindent
\\[2mm]
{\bf end}
\end{minipage}
}
\vgap
\vgap

Some remarks about the modified HPE framework are in order. 
First, it is an instance of the non-Euclidean HPE framework of \cite{gonccalves2016extending}  with $\lambda_k=1, \varepsilon_k=0$ and   
$(d w)_{z}(z') = (1/2)\|z-z'\|_{M}^{2}$, $\forall z,z'\in \Z$.  Second, the way to obtain  $( z_k, \tilde{z}_k,\eta_k)$ will depend on the particular instance of the framework and properties of the operator $T$. In section~\ref{subsec:ADMM&HPE}, we will show that a generalized ADMM can be seen as an instance of the HPE framework specifying, in particular, how this triple  $(z_k,\tilde{z}_k, \eta_k)$ can be obtained.
Third, if $M$ is positive definite and $\sigma=\eta_0 = 0$, then   \eqref{breg-cond1}   implies that $\eta_k = 0$ and 
$z_k =\tilde{z}_k$ for every $k$, and hence that $M(z_{k-1}-z_{k}) \in  T({z}_k)$ in view of \eqref{breg-subpro}.
Therefore, the HPE error conditions \eqref{breg-subpro}-\eqref{breg-cond1} can be viewed as a relaxation of an iteration of the exact  proximal point method. 

In the following, we  present pointwise and ergodic  iteration-complexity  results for the modified HPE framework. 
Let $d_0$ be the distance of $z_0$ to the solution set of $T^{-1}(0)$, i.e., 
\begin{equation}\label{d0HPE}
d_{0} = \inf \{ \|z^*-z_{0}\|_{M}^{2} : z^* \in T^{-1}(0)\}.
\end{equation}
For  convenience of the reader and completeness, the proof of the next two  results are presented in Appendix~\ref{sec:hpe_Analysis}. 

\begin{theorem} {\rm ({\bf Pointwise convergence of the  HPE}) }\label{th:pointwiseHPE}
Consider the sequence $\{(z_k,\tilde z_k,\eta_k)\}$  generated by the modified HPE framework with $\sigma <1 $.
 Then, for every $k \ge 1$, there hold $0\in M (z_{k}-z_{k-1})+T(\tz_k)$ and  there exists $i\leq k$ such that
  \[
  \|z_{i}-z_{i-1}\|_M \leq
\frac{1}{\sqrt{k}} \sqrt{\frac{2(1+\sigma)d_0+4\eta_0}{1-\sigma}},
\]
where $d_0$ is as defined in \eqref{d0HPE}.
\end{theorem}
%

Next, we present the  ergodic convergence  of the modified HPE framework. Before, let us consider the following ergodic sequences
\begin{equation}\label{SeqErg}
 \tilde z^a_{k} = \frac{1}{k}\sum_{i=1}^k\tilde z_i,\quad r^a_{k} = \frac{1}{k}\sum_{i=1}^k(z_{i}-z_{i-1}), \quad
\varepsilon^a_{k} := \frac{1}{k} \sum_{i=1}^k\inner{M(z_{i}-z_{i-1})}{\tilde z^a_{k}-\tilde{z}_i },\quad \forall k\geq 1.
\end{equation}
\begin{theorem} {\bf (Ergodic convergence of the  HPE)}\label{th:ergHPE}
Consider the ergodic sequence $\{(\tilde z_k^a, r^a_{k},\varepsilon^a_{k})\}$ as in \eqref{SeqErg}.  
For every $k\geq 1$, there hold $\varepsilon^a_{k}\geq 0$, $0\in Mr^a_k + T^{[\varepsilon_k^a]}(\tz^a_k)$ and
\begin{align*}
 \|r^a_k\|_M &\le \frac{2\sqrt{d_0+\eta_0}}{k}, \quad
 \varepsilon^a_{k} \leq 
\frac{3\left[ 3 (d_{0} +\eta_0) + \sigma \rho_k\right]}{2k},
\end{align*}
where
 \begin{equation}\label{def:rho}
\rho_k := \displaystyle\max_{i=1,\ldots,k}\|\tilde z_{i}-z_{i-1}\|_{M}^{2},
\end{equation}
and $d_0$ is as defined in \eqref{d0HPE}.
Moreover, the sequence $\{\rho_k\}$ is bounded under either
one of the following situations:
\begin{itemize}
\item[(a)]
$\sigma<1$, in which case
\begin{equation} \label{def:tauk}
\rho_k \le \frac{d_0+\eta_0}{1-\sigma};
\end{equation}
\item[(b)]
$\Dom\, T$ is bounded, in which case
\[
\rho_k \leq  2[ d_0 +\eta_0+ D]  \label{def:tauk1},
\]
where $D := \sup \{ \|y'- y\|_{M}^{2} :  y,y' \in \Dom \,T\}$.

\end{itemize}
\end{theorem}

If $\sigma<1$ or  $\Dom \,T$ is bounded, it follows from Theorem~\ref{th:ergHPE}  that $\{\rho_k\}$ is bounded and hence  $\max\{ \|r^a_k\|_M,\varepsilon^a_k\}=\mathcal{O}(1/k)$.
However,  it may happen that the sequence $\{\rho_k\}$ is bounded even when $\sigma=1$.  Indeed, in the next section, we will present  a generalized ADMM which is an instance of the modified HPE framework satisfying this case (see Lemma~\eqref{lem:aux3}).


\section{The generalized ADMM and its convergence rates}\label{sec:proximal ADMM_proof}
The main goal of this section is to describe the generalized ADMM for solving  \eqref{optl}
and  present  pointwise and ergodic   iteration-complexity results for it.
Our   iteration-complexity  bounds  are obtained  by showing that this ADMM variant is a special case   of the modified HPE framework of Section \ref{sec:smhpe}.


Throughout this section, we assume that:
\begin{itemize}
  \item[\bf A1)] the
  problem \eqref{optl} has an optimal solution $(x^*,y^*)$ and an associated Lagrange multiplier $\gamma^*$, or equivalently,
   the inclusion
\begin{equation} \label{FAB}
0\in T(x,y,\gamma) := \left[ \begin{array}{c}  \partial f(x)- A^{*}\gamma \\ \partial g(y)- B^{*}\gamma \\ Ax+By-b
\end{array} \right] 
\end{equation}
has a solution $(x^*,y^*,\gamma^*)$;
\end{itemize}

\subsection{The generalized ADMM} \label{sub:genr:ADMM}

In this subsection, we recall the  generalized ADMM first proposed by  Eckstein and Bertsekas (see \cite{MR1256136,MR1168183,GADMM2015}) for solving  \eqref{optl}.

\vgap
\vgap
\noindent
\fbox{
\begin{minipage}[h]{6.4 in}
{\bf Generalized ADMM}
\\[2mm]
(0) Let an initial point $(x_0,y_0,\gamma_0) \in \R^{n}\times \R^{p}\times \R^{m}$, a penalty parameter $\beta>0$, 
a relaxation factor $\alpha\in(0,2]$, and  symmetric positive semidefinite matrices $H_1 \in \R^{n \times n}$ and $H_2\in \R^{p \times p}$ 
be given, and set $k=1$;
\\[2mm]
(1)    compute an optimal solution $x_k \in \R^{n}$  of the subproblem
\begin{equation} \label{def:tsk-admm}
\min_{x \in \R^{n}} \left \{ f(x) - \inner{ \gamma_{k-1}}{Ax} +
\frac{\beta}{2} \| Ax+B y_{k-1} - b \|^2+\frac{1}{2}\|x- x_{k-1}\|_{H_1}^2 \right\}
\end{equation}
and compute an optimal solution $y_k\in \R^{p}$ of the subproblem
\begin{equation} \label{def:tyk-admm}
\min_{y \in \R^{p}} \left \{ g(y) - \inner{ \gamma_{k-1}}{By} +
\frac{\beta}{2} \| \alpha(Ax_{k}+B y_{k-1} - b)+B(y-y_{k-1}) \|^2 +\frac{1}{2}\|y- y_{k-1}\|_{H_2}^2\right\};
\end{equation}
\\[2mm]
(2) set 
\begin{equation}\label{admm:eqxk}
\gamma_k = \gamma_{k-1}-\beta[\alpha(Ax_{k}+B y_{k-1} - b)+B(y_{k}-y_{k-1})]
\end{equation}
and $k \leftarrow k+1$, and go to step~(1).
\\[2mm]
{\bf end}
\end{minipage}
}
\\[2mm]

The generalized ADMM  has different features depending on the choices of the operators $H_1$, $H_2$,  and the relaxation factor $\alpha$. 
For instance, by taking $\alpha=1$ and  $(H_1,H_2)=(0,0)$, it reduces to the standard ADMM, and  $\alpha=1$ and $(H_1,H_2)=(\tau_1 I_n-\beta A^{*}A,\tau_2 I_p-\beta B^{*}B)$ with  $\tau_1>\beta\|A^*A\|$ and $\tau_2>\beta\|B^*B\|$, it reduces to the linearized ADMM.
The latter method basically consists of canceling the quadratic terms $(\beta/2)\|Ax\|^2$ and $(\beta/2)\|By\|^2$ in \eqref{def:tsk-admm} and \eqref{def:tyk-admm}, respectively. More specifically, the subproblems  \eqref{def:tsk-admm} and \eqref{def:tyk-admm} become
$$
\min_{x \in \R^{n}} \left \{ f(x) - \inner{ \gamma_{k-1}-\beta(Ax_{k-1}+B y_{k-1} - b)}{Ax} +\frac{\tau_1}{2}\|x- x_{k-1}\|^2 \right\},
$$
$$
\min_{y \in \R^{p}} \left \{ g(y) - \inner{ \gamma_{k-1}-\alpha\beta(Ax_{k}+B y_{k-1} - b)}{By} +\frac{\tau_2}{2}\|y- y_{k-1}\|^2\right\}.
$$
In many applications,  the  above subproblems are  much easier to solve or  even have  closed-form solutions (see   \cite{HeLinear,Wang2012,Yang_linearizedaugmented} for more details). 
We also mention that depending on the structure of problem~\eqref{optl}, other choices of $H_1$ and $H_2$  may be recommended; see, for instance,  \cite{Deng1} (although the latter reference  considers $\alpha=1$, it is clear that the same discussion regarding the choices of $H_1$ and $H_2$ holds  for arbitrary $\alpha\in(0,2)$). 
The generalized ADMM with over-relaxation parameter ($\alpha >1$)  may present computational advantages over the standard ADMM (see, for example, \cite{MR1256136}).
\subsection{The generalized ADMM as an instance of the modified HPE framework}\label{subsec:ADMM&HPE}
Our aim in this subsection is to show that the  generalized ADMM is an instance  of the modified HPE framework for solving the inclusion problem \eqref{FAB} and, as a by-product,    pointwise and ergodic  iteration-complexity bounds results for the generalized ADMM will be presented in Subsection~\ref{PointErg}.

Let us first  introduce the elements required by the setting of Subsection \ref{sec:smhpe}.
 Consider  the vector space $\Z:=\R^{n}\times\R^{p}\times \R^{m}$, the linear operator   
\begin{equation}\label{def:matrixM}
M:=\left[ 
\begin{array}{ccc} 
H_1 &0&0\\
0&(H_2+\frac{\beta}{\alpha} B^*B)&\frac{(1-\alpha)}{\alpha}B^*\\[2mm]
0&\frac{(1-\alpha)}{\alpha}B&\frac{1}{\alpha\beta}I_{m}
\end{array} \right],
\end{equation}
and the quantity 
\begin{equation}\label{def:d0admm}
d_0:=\inf_{(x,y,\gamma) \in T^{-1}(0)} \left\{\|(x-x_0,y-y_0,\gamma-\gamma_0)\|^{2}_{M}\right\}.
\end{equation}
It is easy to verify that $M$  is a symmetric positive semidefinite matrix for every $\beta>0$ and $\alpha\in (0,2]$.
Let $\{(x_k,y_k,\gamma_k)\}$ be the sequence generated by the generalized ADMM. In order to simplify some relations in the results below, define the sequence 
  $\{(\Delta x_k,\Delta y_k,\Delta \gamma_k,\tilde{\gamma}_k)\}$ as
\begin{equation}\label{xtilde} 
\Delta x_k = x_{k}-x_{k-1},\quad\Delta y_k=y_{k}-y_{k-1},\quad\Delta\gamma_k=\gamma_{k}-\gamma_{k-1},\quad 
\tilde{\gamma}_{k}=\gamma_{k-1}-\beta(Ax_{k}+By_{k-1}-b)
\end{equation}
for every $k\geq 1$.  
  
We next present  two technical results on the generalized ADMM. 

\begin{lemma} \label{pr:aux}
Let $\{(x_k,y_k,\gamma_k)\}$ be  generated by the generalized ADMM and consider
$\{(\Delta x_k, \Delta y_k, \Delta\gamma_k, \tilde{\gamma}_k)\}$   as in \eqref{xtilde}. Then, for every $k\geq 1$,
\begin{align}
\tilde{\gamma}_k-\gamma_{k-1}&=\frac{1}{\alpha}\left[\Delta \gamma_k+\beta B\Delta y_k\right], \label{aux.3}\\[3mm]
0&\in H_1\Delta x_k+\left[ \partial f(x_k)-A^*\tilde{\gamma}_k\right],  \label{aux.0}\\[3mm]
0&\in(H_2+\frac{\beta}{\alpha} B^*B)\Delta y_k+\frac{(1-\alpha)}{\alpha}B^*\Delta \gamma_k+\left[\partial g(y_k)-B^*\tilde{\gamma}_k\right],\label{aux.2}\\[3mm]
0&=\frac{(1-\alpha)}{\alpha}B\Delta y_k+\frac{1}{\alpha\beta}\Delta \gamma_k+\left[Ax_k+By_k-b \right].\label{aux.1}
\end{align}
As a consequence,  $z_k:=(x_k,y_k, \gamma_k)$ and $\tilde z_k:=(x_k,y_k,\tilde \gamma_k)$  
satisfy the inclusion~\eqref{breg-subpro}  with $T$ and $M$ as in \eqref{FAB} and \eqref{def:matrixM}, respectively.
\end{lemma}

\begin{proof}
It follows from  definitions of ${\gamma}_k$ and $\tilde{\gamma}_k$ in \eqref{admm:eqxk} and \eqref{xtilde}, respectively,  that
\begin{align*}
\frac{1}{\alpha}(\gamma_k-\gamma_{k-1})+\frac{\beta}{\alpha}B(y_k-y_{k-1})= -\beta(Ax_{k}+By_{k-1}-b)=\tilde{\gamma}_{k}-\gamma_{k-1},
\end{align*}
which, combined with definitions of   $\Delta y_k$ and $\Delta \gamma_k$  in \eqref{xtilde}, proves  \eqref{aux.3}.
From the optimality condition for  \eqref{def:tsk-admm}, we have
\begin{equation*}
0 \in \partial f(x_k)- A^*({\gamma}_{k-1}-\beta(Ax_{k}+By_{k-1}-b))+H_1(x_k-x_{k-1}),
\end{equation*}
which, combined with definitions of $\tilde{\gamma}_k$ and $\Delta x_k$ in \eqref{xtilde}, yields  \eqref{aux.0}.
Similarly, from the optimality condition for \eqref{def:tyk-admm} and  definitions of ${\gamma}_k$ and $\Delta y_k$ in \eqref{admm:eqxk} and  \eqref{aux.0}, respectively, we obtain
\begin{align}\nonumber
0 &\in   \partial g(y_k)-B^*\left[\gamma_{k-1}-\beta\alpha (A{x}_{k}+By_{k-1}-b)+\beta B(y_{k}-y_{k-1})\right]+H_2(y_k- y_{k-1}) \\[2mm]\label{in:ang:suby}
& = \partial g(y_k)-B^*\gamma_k+H_2\Delta y_k.
\end{align}
On the other hand, note that  \eqref{aux.3} implies  that     
$$
\gamma_k=\tilde{\gamma}_k+(\gamma_{k}-\gamma_{k-1})-(\tilde{\gamma}_k-\gamma_{k-1})=\tilde{\gamma}_k-\frac{(1-\alpha)}{\alpha}\Delta\gamma_{k}-\frac{\beta}{\alpha}B\Delta y_{k},
$$
which in turn, combined with \eqref{in:ang:suby}, gives \eqref{aux.2}.
The relation  \eqref{aux.1} follows immediately from \eqref{admm:eqxk}.

 Now, the last statement  of the lemma follows directly by \eqref{aux.0}--\eqref{aux.1} and definitions of $T$ and $M$ given in \eqref{FAB} and \eqref{def:matrixM}, respectively.
\end{proof}


 
\begin{lemma}\label{lem:deltak} The sequences $\{\Delta y_k\}$ and $\{\Delta \gamma_k\}$ defined in \eqref{xtilde} satisfy 
\begin{equation}\label{lem:Deltayk}
2\inner{B\Delta y_1}{\Delta \gamma_1}\geq \|\Delta y_1\|_{H_2}^2 - 4d_0, \qquad 2\langle B\Delta y_k,\Delta \gamma_k \rangle
\geq \|\Delta y_k\|_{H_2}^2-\|\Delta y_{k-1}\|_{H_2}^2 \quad \forall k\geq 2,
\end{equation} 
where $d_0$ is as in \eqref{def:d0admm}.
\end{lemma}
\begin{proof}
 Let  a point  $z^*:=(x^*,y^*,\gamma^*)$ be  such that $0\in T(x^*,y^*,\gamma^*)$ (see assumption {\bf A1}) and consider $z_i:=(x_i,y_i,\gamma_i)$, $i=0,1$.  First, note that 
$$
0\leq\frac{\beta}{\alpha}\|B\Delta y_1\|^2+\frac{2}{\alpha}\langle B\Delta y_1,\Delta\gamma_1 \rangle+  \frac{1}{\alpha\beta} \|\Delta \gamma_1\|^2,
$$  
where $\Delta y_1$ and $\Delta \gamma_1$ are as in \eqref{xtilde}.
Hence, by adding $\|\Delta y_1\|_{H_2}^2-2\inner{B\Delta y_1}{\Delta \gamma_1}$ to both sides of the above inequality, we obtain
\begin{align} 
\|\Delta y_1\|_{H_2}^2-2\inner{B\Delta y_1}{\Delta \gamma_1} &\leq \|\Delta y_1\|_{H_2}^2+
\frac{\beta}{\alpha}\|B\Delta y_1\|^2+2\frac{(1-\alpha)}{\alpha}\inner{B\Delta y_1}{\Delta \gamma_1}+
\frac{1}{\alpha\beta}\|\Delta \gamma_1\|^2 \nonumber\\
 &\leq \|z_1-z_0\|^{2}_{M}\leq 2\left(\|z^{*}-z_1\|^{2}_{M}+\|z^{*}-z_0\|^{2}_{M}\right),\label{eq_000000012}
\end{align}
where $M$ is as in \eqref{def:matrixM} and the last inequality is a consequence of \eqref{fact} with $Q=M$.
On the other hand, taking  $ \tilde{z}_1=(x_1,y_1,\tilde{\gamma}_1)$,
Lemma~\ref{pr:aux} implies that  $(z_0, z_1,\tilde{z}_1)$ satisfies  \eqref{breg-subpro}  with $T$ and $M$ as in \eqref{FAB} and \eqref{def:matrixM}, respectively; namely, $M(z_0-z_1) \in T(\tz_1)$. Hence, since  $0 \in T(z^*)$ and $T$ is monotone, we obtain  $\langle M(z_0-z_1), \tilde{z}_1 - z^*\rangle\geq 0$. Thus, it follows  that 
\begin{align}\nonumber
 \|z^{*}-z_{1}\|^{2}_{M}-\|z^{*}-z_{0}\|^{2}_{M} &= \|(z^{*}-\tilde{z}_1)+(\tilde{z}_1-z_1)\|^{2}_{M}-\|(z^{*}-\tilde{z}_1)+(\tilde{z}_1-z_{0})\|^{2}_{M}\\
 \nonumber
& =  \|\tilde{z}_1-z_{1}\|^{2}_{M}+2\langle M(z_{0}-z_1),z^{*}-\tilde{z}_1 \rangle-\|\tilde{z}_1-z_{0}\|^{2}_{M}\\ \label{ineq_lema007d1d0}
& \leq \|\tilde{z}_1-z_{1}\|^{2}_{M}-\|\tilde{z}_1-z_{0}\|^{2}_{M}.
\end{align}
Combining  \eqref{xtilde} and \eqref{aux.3}, we have $\tilde{\gamma}_1-\gamma_{1}=[(1-\alpha)\Delta\gamma_1+\beta B\Delta y_1]/\alpha$. Hence, using the definitions of $M$, $z_1$ and $\tilde z_1$, we obtain
\begin{align}\label{equ:lem:est1}
\|\tilde{z}_1-z_{1}\|^{2}_{M}=
\frac{1}{\alpha\beta}\|\tilde{\gamma}_1-\gamma_{1}\|^{2}= \frac{\beta}{\alpha^3}\|B\Delta y_1\|^2+2\frac{(1-\alpha)}{\alpha^3}\inner{B\Delta y_1}{\Delta\gamma_1}
+\frac{(1-\alpha)^2} {\alpha^3\beta}\left\|\Delta\gamma_1\right\|^2 \nonumber
\end{align}
and 
\begin{align*}
\|\tilde{z}_1-z_{0}\|_{M}^{2}&\geq \frac{\beta}{\alpha} 
\|B (y_1-y_0)\|^2+\frac{2(1-\alpha)}{\alpha}\inner{ B(y_1-y_0)}{ \tilde{\gamma}_1-\gamma_{0}}
+\frac{1}{\alpha\beta}\|\tilde{\gamma}_1 -\gamma_{0}\|^2\\ \nonumber
&= \left(\frac{\beta}{\alpha}+2\frac{(1-\alpha)\beta}{\alpha^2}+\frac{\beta}{\alpha^3} \right)\|B\Delta y_1\|^2+
2\left(\frac{(1-\alpha)}{\alpha^2}+\frac{1}{\alpha^3}\right)\inner{ B\Delta y_1}{ \Delta \gamma_1}+
\frac{1}{\alpha^3\beta}\left\|\Delta\gamma_1\right\|^2,
\end{align*}
where the last equality  is due to \eqref{xtilde} and \eqref{aux.3}.
Hence, it is easy to see that 
\begin{align}\nonumber
\|\tilde{z}_1-z_1\|^{2}_{M}-\|\tilde{z}_1-z_0\|^{2}_{M}&\leq 
\frac{(\alpha-2)}{\alpha^2}\left\|\sqrt{\beta} B\Delta y_1 
+\frac{1}{\sqrt{\beta}}\Delta\gamma_1\right\|^2\leq 0.
\end{align}
Thus, it follows from  \eqref{ineq_lema007d1d0} that
$$
\|z^*-z_1\|^{2}_{M}\leq \|z^*-z_0\|^{2}_{M},
$$
which, combined with \eqref{eq_000000012}, yields 
$$
\|\Delta y_1\|_{H_2}^2-2\inner{B\Delta y_1}{\Delta \gamma_1} \leq 4\|z^{*}-z_0\|^{2}_{M}.
$$
Therefore, the first inequality in \eqref{lem:Deltayk} follows from definition of $d_0$ (see  \eqref{def:d0admm}) and the fact that $z^*\in T^{-1}(0)$ is arbitrary.

Let us now prove the second inequality in \eqref{lem:Deltayk}.
First, from the optimality condition of  \eqref{def:tyk-admm} and  \eqref{admm:eqxk},  we obtain
$$
 B^{*}\gamma_{j}-H_{2}(y_{j}-y_{j-1})\in \partial g(y_j)\qquad \forall j\geq1.
$$
For every $k\geq 2$, using the previous inclusion for $j=k-1$ and $j=k$, it follows from the monotonicity of the subdifferential of $g$ that
\[
\left \langle B^*(\gamma_k-\gamma_{k-1})- H_{2}(y_k-y_{k-1})+ H_{2}(y_{k-1}-y_{k-2}),y_k-y_{k-1}\right\rangle\geq0,
\]
which, combined with \eqref{xtilde}, yields
\[
\langle B\Delta y_k,\Delta \gamma_k \rangle
\geq\|\Delta y_k\|_{H_2}^2- \langle H_{2}\Delta y_{k-1},\Delta y_k\rangle \qquad \forall k\geq 2.
\]
To conclude the proof, use the  relation \eqref{fact} with $Q=H_{2}$.
\end{proof}  
 
The following theorem  shows that the generalized ADMM is an instance of the modified HPE framework. 
Let us consider the following quantity:
\begin{equation}\label{def:sigmaalpha}
\sigma_{\alpha}= \frac{1}{1+\alpha(2-\alpha)}.
\end{equation} 

Note that $\sigma_{2}=1$, and  for any  $\alpha\in(0,2)$ we have $\sigma_{\alpha}\in(0,1)$.

\begin{theorem}\label{th:admm_hpe} Let 
$\{(x_k,y_k,\gamma_k)\}$  be generated by the generalized ADMM  and consider $\{(\Delta y_k,\tilde{\gamma}_k)\}$  and  $\sigma_{\alpha}$ as in \eqref{xtilde} and \eqref{def:sigmaalpha}, respectively.
Define 
\begin{equation}\label{def:zk}
z_{k-1}=(x_{k-1},y_{k-1},\gamma_{k-1})\qquad\tilde{z}_k=(x_k,y_k,\tilde{\gamma}_k), \qquad \forall\, k\geq 1,
\end{equation}
and 
\begin{equation}\label{eta}
\eta_0=\frac{4(2-\alpha)\sigma_{\alpha}}{\alpha}d_0,\qquad\eta_k=\frac{(2-\alpha)\sigma_{\alpha}}{\alpha}\|\Delta y_k\|_{H_2}^{2} \qquad\forall\,k\geq 1,
\end{equation}
where $d_0$ is as in \eqref{def:d0admm}. Then, the sequence $\{(z_k,\tilde{z}_k,\eta_k)\}$ is an instance of the modified HPE framework  applied for solving  \eqref{FAB}, where  $\sigma:=\sigma_\alpha$ and $M$ is as in~\eqref{def:matrixM}.
\end{theorem}
\begin{proof}
The inclusion \eqref{breg-subpro} follows from the last statement in Lemma~\ref{pr:aux}.
Let us now show that \eqref{breg-cond1} holds. Using 
\eqref{xtilde}, \eqref{aux.3} and \eqref{def:zk}, we obtain
\begin{align}
\|\tilde{z}_k -z_{k}\|_{M}^{2}&=\frac{1}{\alpha\beta}
\|\tilde{\gamma}_k-\gamma_k\|^2= \frac{1}{\alpha\beta}
\left\|\frac{1}{\alpha}\left[(1-\alpha){\Delta \gamma_k}+\beta B\Delta y_k\right]\right\|^2  \nonumber \\
 & = \frac{1}{\alpha^3\beta}\left[(1-\alpha)^2\|\Delta \gamma_k\|^2 +2(1-\alpha)\beta\inner { B\Delta y_k}{\Delta \gamma_k}+\beta^2\|B\Delta y_k\|^2\right] \label{aux.102}.
\end{align}    
Also, \eqref{xtilde}  and  \eqref{def:zk}  imply that
{\small \begin{equation}\label{ad90}
\|\tilde{z}_k-z_{k-1}\|_{M}^{2}=\|\Delta x_k\|_{H_1}^2+\|\Delta y_k\|_{H_2}^2+ \frac{\beta}{\alpha} 
\|B\Delta y_k\|^2+2\frac{(1-\alpha)}{\alpha}\inner{ B\Delta y_k}{ \tilde{\gamma}_k-\gamma_{k-1}}
+\frac{1}{\alpha\beta}\|\tilde{\gamma}_k -\gamma_{k-1}\|^2.
\end{equation}        }
It follows from \eqref{aux.3} that  
\begin{align*}
\frac{1}{\alpha\beta}\|\tilde{\gamma}_k-\gamma_{k-1}\|^2 &=
\frac{1}{\alpha^3\beta}\left[\|\Delta \gamma_k\|^{2}+2\beta\inner{B\Delta y_k}{\Delta\gamma_{k}}+\beta^2\|B\Delta y_k\|^2\right],\\
2\frac{(1-\alpha)}{\alpha}\inner{B\Delta y_k}{\tilde{\gamma}_k-\gamma_{k-1}}&=2\frac{(1-\alpha)}{\alpha^2}\left[\inner{ B\Delta y_k}{\Delta\gamma_k}+\beta\|B\Delta y_k\|^2\right]
\end{align*} 
which, combined with \eqref{ad90}, yields 
\begin{align}\nonumber
\|\tilde{z}_k-z_{k-1}\|_{M}^{2} = \|\Delta x_k\|_{H_1}^2&+\|\Delta y_k\|_{H_2}^2 +
\left(\frac{\beta}{\alpha}+2\frac{(1-\alpha)\beta}{\alpha^2}+\frac{\beta}{\alpha^3}\right)\|B\Delta y_k\|^2\\\label{aux.103}
&+2\left(\frac{(1-\alpha)}{\alpha^2}+\frac{1}{\alpha^3}\right)\inner{B\Delta y_k}{\Delta\gamma_k}
+\frac{1}{\alpha^3\beta}\|\Delta \gamma_k\|^2.
\end{align}
Therefore, combining \eqref{aux.102} and \eqref{aux.103}, it is easy to verify that
\begin{align*}
\sigma_{\alpha}\|\tilde{z}_k-z_{k-1}\|_{M}^{2}&-\|\tilde{z}_k-z_{k}\|_{M}^{2}\\
 &= \sigma_{\alpha}\|\Delta x_k\|_{H_1}^2+\sigma_{\alpha}\|\Delta y_k\|_{H_2}^2+
2\frac{(2-\alpha)\sigma_{\alpha}}{\alpha}\inner{B\Delta y_k}{\Delta\gamma_k}
+\frac{(2-\alpha)^{2}\sigma_{\alpha}}{\alpha\beta}\|\Delta\gamma_k\|^2\\
&\geq 2\frac{(2-\alpha)\sigma_{\alpha}}{\alpha}\inner{B\Delta y_k}{\Delta\gamma_k}\geq \eta_k - \eta_{k-1}\qquad \forall\,k\geq 1,
\end{align*}
where $\sigma_\alpha$ is as in \eqref{def:sigmaalpha}, 
and the last inequality is due to \eqref{lem:Deltayk} and \eqref{eta}.  Therefore, \eqref{breg-cond1} holds, and then 
we conclude that the sequence 
$\{(z_k,\tilde{z}_k,\eta_k)\}$ is an instance of the modified HPE framework.
\end{proof}

\subsection{Iteration-complexity bounds for the generalized ADMM}\label{PointErg}
In this subsection, we study pointwise and ergodic  iteration-complexity bounds for the generalized ADMM. We start by presenting a pointwise  bound
under the assumption that the relaxation parameter $\alpha$ belongs to $(0,2)$. Then, we consider an auxiliary result which is used to show that the sequence $\{\rho_k\}$, as defined in Theorem \ref{th:ergHPE} with $\{z_k\}$ and $\{\tilde{z}_k\}$ as in \eqref{def:zk}, is bounded even in the extreme case in which $\alpha=2$. This latter result is then used to present the ergodic bounds of the generalized ADMM for any $\alpha\in (0,2]$.

\begin{theorem} {\bf (Pointwise convergence   of the generalized ADMM)} \label{th:pointwise} 
Let  $\{(x_k,y_k,\gamma_k)\}$ be  generated by the generalized ADMM  with $\alpha\in (0,2)$  
and consider the sequence $\{(\Delta x_k,\Delta y_k, \Delta\gamma_k,\tilde{\gamma}_k)\}$ as in \eqref{xtilde}.
Then, for every $k\geq 1$, 
\begin{equation}\label{eq:th_incADMMtheta1} 
0\in M\left(
\begin{array}{c} 
\Delta x_k\\
\Delta y_k\\
\Delta\gamma_k
\end{array} 
\right) + 
\left( 
\begin{array}{c} 
\partial f(x_k)- A^*\tilde{\gamma}_k\\[1mm]  
\partial g(y_k)- B^*\tilde{\gamma}_k\\[1mm]  
Ax_k+By_k-b
\end{array} \right)
\end{equation}
and   there exists  $i\leq k$ such that
\[
\|(\Delta x_i,\Delta y_i,\Delta \gamma_i)\|_{M} \leq \frac{1}{\sqrt{k}} \sqrt{\frac{2[\alpha(1+\sigma_{\alpha})+8(2-\alpha)\sigma_{\alpha}]d_{0}}
{\alpha(1-\sigma_{\alpha})}},
\]
where  $M$, $d_0$,  and $\sigma_{\alpha}$ are as \eqref{def:matrixM}, \eqref{def:d0admm} and \eqref{def:sigmaalpha}, respectively.
\end{theorem}
\begin{proof}
Since $\sigma_{\alpha} \in (0,1)$ for any  $\alpha\in (0,2)$ (see \eqref{def:sigmaalpha}),
we obtain  by combining Theorems~\ref{th:pointwiseHPE} and \ref{th:admm_hpe}   that  \eqref{eq:th_incADMMtheta1}  
holds and there exists $i\leq k$ such that 
\[
\norm{(\Delta x_i,\Delta y_i,\Delta \gamma_i)}_{M}\leq
\frac{1}{\sqrt{k}} \sqrt{\frac{2(1+\sigma_{\alpha})d_0+4\eta_0}{1-\sigma_{\alpha}}}.
\]
Hence,  to conclude the proof use the definition of  $\eta_0$ given in  \eqref{eta}.  
\end{proof}

For a given tolerance $\varepsilon>0$, Theorem~ \ref{th:pointwise} implies  that  in   at most $\mathcal{O}(1/\varepsilon^2)$ iterations, the G-ADMM  obtains  an ``$\varepsilon$-approximate" solution $( x,  y,\gamma)$  and a residual $ v$ of \eqref{LagInclusion_Intro} satisfying
\[
Mv \in T(x,y,\gamma), \quad \| v\|_M\leq \varepsilon,
\]
where $T$ and $M$ are as \eqref{FAB} and \eqref{def:matrixM}, respectively.

Next we consider an auxiliary result which will be used to obtain ergodic   iteration-complexity bounds  for the generalized ADMM.
\begin{lemma}\label{lem:aux3}
Let $\{(x_k,y_k,\gamma_k)\}$ be generated by the generalized ADMM and consider 
$\{(\Delta x_k,\Delta y_k, \Delta\gamma_k,\tilde{\gamma}_k)\}$ as in \eqref{xtilde}. Then, the sequence $\{\rho_k\}$ given  in \eqref{def:rho} with $M$ and $\{(z_{k},\tz_k)\}$ as  in  \eqref{def:matrixM} and \eqref{def:zk}, respectively,  satisfies
\[
\rho_k  \leq\frac{4(1+2\alpha)[\alpha+4(2-\alpha)\sigma_{\alpha}]d_{0}}{\alpha^3}\qquad \forall\, k\geq1,
\]
where $d_0$ is as in  \eqref{def:d0admm}.
\end{lemma}
\begin{proof} 
The same argument used to prove \eqref{ad90} and \eqref{aux.103} yields, for every $k\geq 1$, 
\begin{equation}\label{normzz0}
\|\tilde{z}_k-z_{k-1}\|_{M}^{2}=\|\Delta x_k\|_{H_1}^2+\|\Delta y_k\|_{H_2}^2+\xi_k,
\end{equation}
where   {\small
\begin{equation*}
\xi_k:=\frac{\beta}{\alpha^3}\|B\Delta y_k\|^2+ \frac{2(1-\alpha)}{\alpha^3}\inner{B\Delta y_k}{\Delta\gamma_k}  + \frac{1}{\alpha^3\beta}\|\Delta \gamma_k\|^2\ +\frac{(2-\alpha)}{\alpha}\left[\frac{\beta}{\alpha}\|B\Delta y_k\|^2+\frac{2}{\alpha}\inner{B\Delta y_k}{\Delta\gamma_k}\right].
\end{equation*}}
Using the definitions of $M$ and $z_{k}$ given in  \eqref{def:matrixM} and \eqref{def:zk}, respectively, it follow that 
\begin{align}\nonumber
\xi_k&\leq\frac{1}{\alpha^2}\|z_k-z_{k-1}\|_{M}^{2}
+\frac{(2-\alpha)}{\alpha}\left[\frac{\beta}{\alpha}\|B\Delta y_k\|^2+\frac{2}{\alpha}\inner{B\Delta y_k}{\Delta\gamma_k}\right]\\\nonumber
&=\small \frac{1}{\alpha^2}\|z_k-z_{k-1}\|_{M}^{2}+
\frac{(2-\alpha)}{\alpha}\left[\frac{\beta}{\alpha}\|B\Delta y_k\|^2+\frac{2(1-\alpha)}{\alpha}\inner{B\Delta y_k}{\Delta\gamma_k}\right]+
\frac{2(2-\alpha)}{\alpha}\inner{B\Delta y_k}{\Delta\gamma_k}\\\nonumber
&\leq\frac{1}{\alpha^2}\|z_k-z_{k-1}\|_{M}^{2}+ \frac{(2-\alpha)}{\alpha}\|z_k-z_{k-1}\|_{M}^{2}+
\frac{2(1-\alpha)}{\alpha}\inner{B\Delta y_k}{\Delta\gamma_k}+\frac{2}{\alpha}\inner{B\Delta y_k}{\Delta\gamma_k}\\\label{aux:c(alpha)}
&\leq \frac{1+2\alpha-\alpha^2}{\alpha^2}\|z_k-z_{k-1}\|_{M}^{2}+
\frac{2(1-\alpha)}{\alpha}\inner{B\Delta y_k}{\Delta\gamma_k}+\frac{\beta}{\alpha}\norm{B\Delta y_k}^{2}+\frac{1}{\alpha\beta}\norm{\Delta\gamma_k}^{2},
\end{align}
where in the last two inequalities we used the fact that $\alpha\in (0,2]$ and  \eqref{fact} with $Q=I_{m}$, respectively. Combining \eqref{normzz0}, \eqref{aux:c(alpha)} and definitions of $M$ and $z_{k}$, we obtain, for every $k\geq 1$, 
\[
\|\tilde{z}_k-z_{k-1}\|_{M}^{2}\leq \frac{1+2\alpha-\alpha^2}{\alpha^2}\|z_k-z_{k-1}\|_{M}^{2}+\|z_k-z_{k-1}\|_{M}^{2}=
\frac{1+2\alpha}{\alpha^2}\|z_k-z_{k-1}\|_{M}^{2}.
\]                                                                                  
Now, letting  $z^*:=(x^*,y^*,\gamma^*)$ be an arbitrary solution  of  \eqref{FAB}, we obtain  from the last inequality and \eqref{fact} with $Q=M$ that
$$
 \|\tilde{z}_k-z_{k-1}\|_{M}^{2}\leq\frac{2(1+2\alpha)}{\alpha^2}\left[\|{z}^*-z_{k}\|_{M}^{2}+\|{z}^*-z_{k-1}\|_{M}^{2}\right] \qquad \forall k\geq 1.
$$
Since  the generalized ADMM is an instance of the modified HPE framework with $\sigma:=\sigma_{\alpha}$ (see Theorem~\ref{th:admm_hpe} and \eqref{def:sigmaalpha}),  it  follows from the last inequality and  Lemma~\ref{lema_desigualdades}(b)  that 
\[ 
\|\tilde{z}_k-z_{k-1}\|_{M}^{2}\leq \frac{4(1+2\alpha)}{\alpha^2}\left[ \|{z}^*-z_{0}\|_{M}^{2}+\eta_{0}\right] \qquad \forall k\geq 1.
\]
Since $z^*$ is an arbitrary solution  of  \eqref{FAB},  the result follows from the definition of $\rho_{k}$, $d_0$, and $\eta_{0}$ given  in \eqref{def:rho}, \eqref{def:d0admm} and \eqref{eta}, respectively.
\end{proof}
Next  result  presents  $\mathcal{O}(1/k)$ convergence rate for the  ergodic sequence associated to the generalized ADMM.

\begin{theorem} {\bf (Ergodic convergence of the generalized ADMM)}\label{th:ergodicproximal ADMM}
Let $\{(x_k,y_k,\gamma_k)\}$ be the sequence  generated by the generalized ADMM 
and consider $\{(\Delta x_k,\Delta y_k, \Delta\gamma_k,\tilde{\gamma}_k)\}$ as in \eqref{xtilde}.
Define the ergodic sequences as 
\begin{align}\label{eq:jase12}
(x_k^a,y_k^a,\gamma_k^a,\tilde{\gamma}^a_k)&=\frac{1}{k}\sum_{i=1}^k\left( x_i, y_i, \gamma_i,\tilde{\gamma}_i\right), \qquad
(r_{k,x}^{a},r_{k,y}^{a},r_{k,\gamma}^{a})=\frac{1}{k}\sum_{i=1}^k(\Delta x_i,\Delta y_i, \Delta\gamma_i),\\\label{eq:jase1212}
\varepsilon^a_{k,x}&= 
\frac{1}{k}\sum_{i=1}^k\inner{H_{1}\Delta x_{i}-A^{*}\tilde{\gamma}_i}{x_k^a-x_i},\\\label{eq:jase121212}
\varepsilon^a_{k,y}&=\frac{1}{k}\sum_{i=1}^k
\Inner{\left(H_{2}+\frac{\beta}{\alpha}B^*B\right)\Delta y_{i}+\frac{(1-\alpha)}{\alpha}B^*\Delta\gamma_{i}-B^*\tilde{\gamma}_i}{y_k^a-y_i}.
\end{align}
Then, for every $k\geq 1$, there hold $\varepsilon^a_{k,x}\geq 0, \;\varepsilon^a_{k,y}\geq 0$, and 

\begin{equation}\label{eq:th_incADMMtheta<1} 
 0\in M\left( 
\begin{array}{c} 
r_{k,x}^{a}\\[1mm]  
r_{k,y}^{a}\\[1mm]  
r_{k,\gamma}^{a}
\end{array} 
\right) + 
\left( 
\begin{array}{c} 
\partial f_{\varepsilon^a_{k,x}}(x_k^a)- A^*\tilde{\gamma}_k^a\\[1mm]  
\partial g_{\varepsilon^a_{k,y}}(y_k^a)- B^*\tilde{\gamma}_k^a\\[1mm]  
Ax_k^a+By_k^a-b
\end{array} \right),
\end{equation}  
\\[2mm]
\begin{equation}\label{rkaeka}
\|(r^a_{k,x}, r^a_{k,y}, r^a_{k,\gamma})\|_{M} \leq \frac{2\sqrt{c_\alpha d_0}}{k}, \quad
\varepsilon^a_{k,x} + \varepsilon^a_{k,y} \leq \frac{\tilde c_\alpha d_{0}}{k},
\end{equation}
where 
\begin{equation}\label{def:c_alpha}
 c_\alpha:=\frac{\alpha+4(2-\alpha)\sigma_{\alpha}}{\alpha},\quad \tilde c_\alpha:=\frac{3[3\alpha^{2}+4(1+2\alpha)\sigma_{\alpha}][\alpha+
4(2-\alpha)\sigma_{\alpha}]}{2\alpha^{3}},
\end{equation}
and  $M$, $d_0$, and $\sigma_{\alpha}$   
are as in \eqref{def:matrixM}, \eqref{def:d0admm}, and \eqref{def:sigmaalpha}, respectively.
\end{theorem}

\begin{proof}
It follows from  Theorem \ref{th:admm_hpe} that the generalized ADMM is an instance of the modified HPE where $\{(z_{k},\tilde z_k)\}$ is given by \eqref{def:zk}. Moreover, it is easy to see that the quantities $r_k^a$ and  $\varepsilon_k^a$ given in \eqref{SeqErg} satisfy
\begin{equation}\label{eq:hjh}
r_k ^a=(r^a_{k,x}, r^a_{k,y}, r^a_{k,\gamma}),  \quad\varepsilon_k^a=\frac{1}{k}\sum_{i=1}^k
\left[\Inner{M\left(
\begin{array}{c} 
\Delta x_i\\ 
\Delta y_i\\ 
\Delta\gamma_i
\end{array}
\right)}{\left(x^{a}_{k}-x_i,y^{a}_{k}-y_i,\tilde{\gamma}^{a}_{k}-\tilde{\gamma}_i\right)}\right].  
\end{equation}
Hence,  from Theorems~\ref{th:ergHPE} and definition of $\eta_0$  in \eqref{eta}, we have 
\begin{equation}\label{eq:al90}
\|r^a_k\|_{M} \leq \frac{2\sqrt{(\alpha+4(2-\alpha)\sigma_{\alpha})d_0}}{k\sqrt{\alpha}}, \quad 
\varepsilon_k^a
\leq \frac{3[3\alpha^{2}+4(1+2\alpha)\sigma_{\alpha}][\alpha+
4(2-\alpha)\sigma_{\alpha}]d_{0}}{2\alpha^{3} k},
\end{equation} 
where in the last inequality we also used Lemma~\ref{lem:aux3}.
Now, we claim that $\varepsilon_k^a=\varepsilon^a_{k,x}+\varepsilon^a_{k,y}$. Using this claim, \eqref{rkaeka} follows immediately from  \eqref{def:c_alpha} and  \eqref{eq:al90}. Hence, to conclude the proof of the theorem,  it just  remains to prove the above claim. To this end, note that   \eqref{eq:jase1212} and \eqref{eq:jase121212} yield
\begin{align}\nonumber
\varepsilon^a_{k,x}+\varepsilon^a_{k,y}&=\frac{1}{k}\sum_{i=1}^{k}\left[\Inner{H_{1}\Delta x_i}{x^{a}_{k}-x_i}+
\Inner{\left(H_{2}+\frac{\beta}{\alpha}B^*B\right)\Delta y_{i}+\frac{\left(1-\alpha\right)}{\alpha}B^*\Delta\gamma_{i}}{y^{a}_{k}-y_i}\right]\\\label{eq:somaepsilon}
&+\frac{1}{k}\sum_{i=1}^k\left\inner{A\left(x^{a}_{k}-x_i\right)+B\left(y^{a}_{k}-y_i\right)}{-\tilde{\gamma}_i}\right].
\end{align}
On the other hand, from \eqref{eq:jase12}, we obtain 
\begin{align*}
\frac{1}{k}\sum_{i=1}^k\Inner{A(x^{a}_{k}-x_i)+B(y^{a}_{k}-y_i)}{-\tilde{\gamma}_i}&=
\frac{1}{k}\sum_{i=1}^k\Inner{Ax^{a}_{k}+By^{a}_{k}-b-(Ax_i+By_i-b)}{\tilde{\gamma}_{k}^{a}-\tilde{\gamma}_i}\\
&=\frac{1}{k}\sum_{i=1}^k\Inner{-(Ax_i+By_i-b)}{\tilde{\gamma}_{k}^{a}-\tilde{\gamma}_i}\\
&=\frac{1}{k}\sum_{i=1}^k\Inner{\frac{(1-\alpha)}{\alpha}B\Delta y_{i}+\frac{1}{\alpha\beta}\Delta\gamma_{i}}{\tilde{\gamma}_{k}^{a}-\tilde{\gamma}_i}
\end{align*}
where the last equality is due to \eqref{aux.1}. Hence, the claim follows  by combining \eqref{eq:somaepsilon},
 and the definitions of $M$ and  $\varepsilon_k^a$  in  \eqref{def:matrixM} and  \eqref{eq:hjh}, respectively.
\end{proof}

For a given tolerance $\varepsilon>0$, Theorem~\ref{th:ergodicproximal ADMM}  implies  that in   at most $\mathcal{O}(1/\varepsilon)$ iterations of the G-ADMM, we obtain  an ``$\varepsilon$-approximate" solution $(\hat x, \hat y, \hat\gamma)$  and a residual $\hat v=(\hat v_1,\hat v_2,\hat v_3)$ of \eqref{LagInclusion_Intro} satisfying
\[
\hat v_1 \in  \partial_{\varepsilon_1} f(\hat x)- A^{*}\hat \gamma,\quad  \hat v_2 \in \partial_{\varepsilon_2} g(\hat y)- B^{*}\hat \gamma, \quad  \hat v_3= A\hat x+B\hat y-b,\quad \|\hat v\|^*_{M}\leq\varepsilon, \quad \varepsilon_1+\varepsilon_2\leq\varepsilon,
\]
where $\|\cdot\|^*_M$ is a dual norm (seminorm) associated to $M$.


\appendix
\section{\bf Appendix (Proofs of Theorems~\ref{th:pointwiseHPE} and \ref{th:ergHPE})} \label{sec:hpe_Analysis}
 
The main goal in this section is to present the proofs of Theorems~\ref{th:pointwiseHPE} and \ref{th:ergHPE}.
Toward this goal, we first  consider a technical lemma.

\begin{lemma}\label{lema_desigualdades}
Let  $\{(z_k,\tilde z_k,\eta_k)\}$ be the sequence generated by the modified HPE framework.
For every $k \geq 1$, the following statements hold:
\begin{itemize}
\item[(a)]
 for every $z \in \Z$, we have
\[
\|z-z_{k}\|^{2}_{M}+\eta_{k}\leq (\sigma-1)\|\tilde{z}_k-z_{k-1}\|^{2}_{M}+\|z-z_{k-1}\|^{2}_{M}+2\langle M(z_{k-1}-z_k),z-\tilde{z}_k\rangle+\eta_{k-1};
\]
\item[(b)] for every  $z^* \in T^{-1}(0)$, we have
\begin{align*}
\|z^*-z_{k}\|^{2}_{M} +\eta_{k} &\leq  (\sigma-1)\|\tilde{z}_k-z_{k-1}\|^{2}_{M}+\|z^*-z_{k-1}\|^{2}_{M}+\eta_{k-1}\leq \|z^*-z_{k-1}\|^{2}_{M} +\eta_{k-1}.
\end{align*}
\end{itemize}
\end{lemma}
\begin{proof} 
(a) Note that, for every $z\in\Z$, 
\begin{align*}
 \|z-z_{k}\|^{2}_{M}-\|z-z_{k-1}\|^{2}_{M} &= \|(z-\tilde{z}_k)+(\tilde{z}_k-z_k)\|^{2}_{M}-\|(z-\tilde{z}_k)+(\tilde{z}_k-z_{k-1})\|^{2}_{M}\\
& =  \|\tilde{z}_k-z_{k}\|^{2}_{M}-\|\tilde{z}_k-z_{k-1}\|^{2}_{M}+2\langle M(z_{k-1}-z_k),z-\tilde{z}_k \rangle,
\end{align*}
which, combined with \eqref{breg-cond1}, proves the desired inequaliy.

(b) Since  $M(z_{k-1}-z_k)\in T(\tz_k)$ and $0 \in T(z^*)$,   we have $\langle M(z_{k-1}-z_k),\tilde{z}_k-z^* \rangle\geq0$.
Hence,  the first inequality in  (b) follows from  (a) with $z=z^*$. Now, the second inequality in (b) follows from the fact that $\sigma\leq 1$.
\end{proof}

\noindent
 {\bf Proof of Theorem~\ref{th:pointwiseHPE}:}
The inclusion $0\in M(z_{k}-z_{k-1})+ T(\tilde{z}_k)$ holds due to \eqref{breg-subpro}. 
 It follows   from \eqref{fact} with $Q=M$ that, for every $j \ge 1$,
\begin{align*}
\|z_{j}-z_{j-1}\|^2_{M}&\le 2(\|\tilde{z}_{j}-z_{j-1})\|^{2}_{M} + \|\tilde{z}_{j}-z_{j}\|^{2}_{M}) 
 \leq  2(1+\sigma)\|\tilde{z}_{j}-z_{j-1}\|^{2}_{M} + 2(\eta_{j-1}-\eta_j)
\end{align*}
where  the last inequality is due to   \eqref{breg-cond1}. 
Now, if $z^* \in T^{-1}(0)$, we obtain from  Lemma \ref{lema_desigualdades}(b)
\[
(1-\sigma)\|\tilde{z}_{j}-z_{j-1}\|^{2}_{M} \leq \|z^*-z_{j-1}\|^{2}_{M} -\|z^*-z_{j}\|^{2}_{M} +\eta_{j-1}-\eta_{j}, \quad \forall\, j\geq 1. 
\]
Combining the last two  inequality, we get
\begin{align*}
(1-\sigma)\sum_{j=1}^k \|z_{j}-z_{j-1}\|^2_{M}&\leq 2(1+\sigma)\left(\|z^*-z_{0}\|^{2}_{M} -\|z^*-z_{k}\|^{2}_{M} +\eta_{0}-\eta_{k}\right)
+2(1-\sigma)(\eta_0- \eta_k)\\
&\leq  2(1+\sigma)\|z^*-z_{0}\|^{2}_{M} + 4\eta_0.
\end{align*}
Hence, as $\sigma< 1 $, we obtain 
\[
\min_{i=1,\ldots,k}\|z_{i}-z_{i-1}\|^2_{M}\leq \frac{1}{k(1-\sigma)}\left(2(1+\sigma)\|z^*-z_{0}\|^{2}_{M}+ 4\eta_0\right).
\]
Therefore, the desired inequality follows from the latter inequality and the definition of $d_0$ in \eqref{d0HPE}.
\hfill{ $\square$}
\\[1mm]

\noindent {\bf Proof of Theorem~\ref{th:ergHPE}:}
The inclusion $0\in Mr^a_k+T^{[\varepsilon_k^a]}(\tz^a_k)$ follows from \eqref{breg-subpro},    \eqref{SeqErg}, and Theorems~\ref{for:trans}(a). 
Using \eqref{SeqErg}, it is easy see that for any $z^{\ast}\in T^{-1}(0)$ 
\[
k r^a_k = z_k-z_0 = (z^*-z_0)+(z_k-z^{*}).
\]
Hence, from inequality \eqref{fact} with $Q=M$ and Lemma~\ref{lema_desigualdades}(b), we have
\[
k^{2} \norm{r^a_k}^{2}_{M} \leq 2( \norm{z^*-z_0}^{2}_{M} + \norm{z^*-z_k}^{2}_{M})\leq 4(\norm{z^*-z_0}_{M}^{2} +\eta_0).
\]
Combining the above inequality with definition of $d_0$, we obtain the bound on $\|r_k^a\|_{M}$. 
Let us  now to prove the bound on $\varepsilon_k^a$. From Lemma~\ref{lema_desigualdades}(a), we have
\[ 
2\sum_{i=1}^k \langle M(z_{i}-z_{i-1}), z-\tilde{z}_i\rangle\leq \norm{z-z_0}_{M}^{2}-\norm{z-z_k}_{M}^{2}+\eta_0-\eta_k
\leq \norm{z-z_0}_{M}^{2}+\eta_0,
\]
for every $z\in\Z$. Letting $z=\tilde z^a_{k}$ and using \eqref{SeqErg},      we get 
\begin{equation}\label{eq:estim_rhok}
2k \varepsilon_k^a\leq \norm{\tilde{z}^{a}_{k}-z_0}_{M}^{2}+\eta_0 
\leq \frac{1}{k}\sum_{i=1}^k\norm{\tilde{z}_{i}-z_0}_{M}^{2}+\eta_0
\leq \max_{i=1,\ldots,k}\norm{\tilde{z}_{i}-z_0}_{M}^{2}  +\eta_0
\end{equation}
where the second inequality is due to convexity of the function $\|\cdot\|^{2}_{M}$, which also 
implies that, for every $i \geq 1$  and $z^* \in T^{-1}(0)$,
\[
\norm{\tilde z_{i}-z_0}^{2}_{M} 
\leq 3\left[\norm{\tilde{z}_{i}-z_i}^{2}_{M}+\norm{z^*-z_i}^{2}_{M}+\norm{z^*-z_0}^{2}_{M} \right].
\]
Hence, using \eqref{breg-cond1} and  twice Lemma \ref{lema_desigualdades}(b), it  follows,  for every $i \geq 1$  and $z^* \in T^{-1}(0)$, that
\begin{align*}
\norm{\tilde z_{i}-z_0}^{2}_{M} &\leq 3 \left[ \sigma\norm{\tilde{z}_{i}-z_{i-1}}^{2}_{M}+\eta_{i-1}+\norm{z^*-z_{i-1}}^{2}_{M}+\eta_{i-1}+
\norm{z^*-z_0}^{2}_{M}\right] \\
& \leq 3\left[ \sigma\norm{\tilde{z}_{i}-z_{i-1}}^{2}_{M}+2(\norm{z^*-z_{i-1}}^{2}_{M}+\eta_{i-1})+\norm{z^*-z_0}^{2}_{M}\right] \\
& \leq 3\left[ \sigma\norm{\tilde{z}_{i}-z_{i-1}}^{2}_{M}+3\norm{z^*-z_0}^{2}_{M}+2\eta_{0}\right], 
\end{align*}
which,  combined with \eqref{eq:estim_rhok} and definitions of $\rho_k$ in \eqref{def:rho}, yields
\[
2k \varepsilon_k^a\leq 3\left[3\norm{z^*-z_0}^{2}_{M} +\sigma \rho_k\right]+7\eta_{0}\leq 3\left[3(\norm{z^*-z_0}^{2}_{M}+\eta_{0})+\sigma\rho_k\right].
\]
Thus,  the bound on $\varepsilon_k^a$ now follows from the definition of the $d_0$ in \eqref{d0HPE}. 

It remains to prove the second part of the theorem.
\\
(a) if $\sigma<1$, then it follows from Lemma \ref{lema_desigualdades}(b), for every $i \ge 1$ and $z^* \in T^{-1}(0)$, that 
\[
(1-\sigma)\norm{\tilde{z}_{i}-z_{i-1}}^{2}_{M}\leq \norm{z^*-z_{i-1}}^{2}_{M}+\eta_{i-1} \leq \norm{z^*-z_{0}}^{2}_{M}+\eta_0.
\]
Hence, in view of definitions of $\rho_k$ and $d_0$, we obtain \eqref{def:tauk}.
\\
(b) If $\Dom\, T$ is bounded, then it follows from 
inequality \eqref{fact} with $Q=M$, and Lemma \ref{lema_desigualdades}(b), for every $i \ge 1$ and $z^* \in T^{-1}(0)$, that 
\[
\norm{\tilde{z}_{i}-z_{i-1}}^{2}_{M}\leq 2 \left[\norm{z^*-z_{i-1}}^{2}_{M}+\norm{\tilde{z}_{i}-z^*}^{2}_{M} \right]
\leq 2 \left[ \norm{z^*-z_0}^{2}_{M} +\eta_0 + D \right]
\]
which, combined with definitions of $\rho_k$ and $d_0$, proves the desired result.   \hfill{ $\square$}


\def\cprime{$'$}

\end{document}